\documentclass[12pt]{amsart}

\usepackage{ucs}

\usepackage{amssymb}
\usepackage{hyperref}
\usepackage{amsthm}
\usepackage{amsmath}
\usepackage{latexsym}
\usepackage[cp1251]{inputenc}
\usepackage{caption}
\usepackage{subcaption}
\usepackage{indentfirst}
\usepackage[left=2.6cm,right=2.6cm,top=3cm,bottom=2.5cm,bindingoffset=0cm]{geometry}
\usepackage{enumerate}

\usepackage[english]{babel}

\makeatletter 
\def\@seccntformat#1{\csname the#1\endcsname. } 
\def\@biblabel#1{#1.}

\makeatother

\title[A new construction of Deza graphs through $\pi$-local fusion graphs]{A new construction of Deza graphs through $\pi$-local fusion graphs of finite simple groups of Lie-type of even characteristic} 
 
 \author{L.Yu. Tsiovkina} 
\address{N.N.~Krasovskii Institute of Mathematics and Mechanics of the Ural Branch of the  Russian Academy of Sciences, 16 S.~Kovalevskaya str., Yekaterinburg, 620108 Russia
}
\email{l.tsiovkina@gmail.com}
 
\date{}

\theoremstyle{theorem} 
\newtheorem{prop}{Proposition}[section]
\newtheorem{thm}[prop]{Theorem}
\newtheorem{lem}[prop]{Lemma}

\theoremstyle{definition}
\newtheorem*{rem}{Remark}

\begin{document}

\vspace{\baselineskip}
\vspace{\baselineskip}

\begin{abstract} 
In this paper, we investigate the  structure of $\pi$-local fusion graphs of some finite simple groups of Lie-type
of even characteristic. We indicate a strong  connection between such graphs and 
other combinatorial objects, as antipodal covers and Deza graphs.
In particular, we find several infinite families of $\pi$-local fusion graphs of finite simple groups of Lie-type of even characteristic that are strictly Deza graphs. 
\end{abstract}
\maketitle
\section*{Introduction}

 Let $G$ be a group with  a   $G$-conjugacy class $X$ of involutions, let
$\omega(G)$ denote  the set of element orders of $G$, and let  $\pi\subseteq\omega(G)$.
Denote by $gF_{\pi}(G,X)$   the graph on $X$, in which involutions $x$ and $y$ are adjacent if
and only if $x\ne y$ and the order of the product $xy$ is contained in $\pi$; we call $gF_{\pi}(G,X)$   a  \textit{$\pi$-product involution graph} of  $G$. 
If the set $\pi$ consists only of odd numbers, then this graph  is called a \textit{$\pi$-local fusion graph} of $G$ and denoted by $F_{\pi}(G,X)$. 
Local fusion graphs of a finite group reflect the structure of its conjugacy classes of involutions, since
  each path between given vertices in such a graph allows one to determine an element that conjugates its end vertices. A study of their various properties  
  is of interest both for abstract group theory and for computational group theory 
  (in particular, when solving the problem of construction of an involution centralizer).
In this regard, it becomes an important step to  investigate local fusion graphs of  finite simple groups. 
For more background and motivation, we refer the reader to \cite{B11}.
 
 To date, main attention has been paid to local fusion graphs of symmetric groups, sporadic simple groups, 
and finite simple groups of Lie-type. These graphs were  studied in \cite{B11,BGR13,RW16,BR16}, where their basic properties, such as diameter and connectivity, were determined. 

In this paper, we are concerned with further structure of $\pi$-local fusion graphs of finite simple groups of Lie-type. 
We indicate a strong  connection between such graphs and intricate 
combinatorial objects, such as antipodal covers and Deza graphs.
In particular, we find several infinite families of $\pi$-local fusion graphs of finite simple groups of Lie-type of even characteristic that are 	strictly Deza graphs. It is worth noting that the proofs 
proposed in the present paper are of a mainly combinatorial nature.

\section{Preliminaries}

 Next we list  some  terminology and  notation that are used in this paper.  
Throughout the paper we consider undirected graphs without loops or multiple edges.   
The distance between  vertices $a$ and $b$ of a graph  $\Gamma$ is denoted by $\partial_\Gamma(a,b)$.  
For a vertex $a$ of a graph  $\Gamma$,  we denote by $\Gamma_i(a)$ the \textit{$i$-neighborhood} of $a$, 
that is the subgraph of $\Gamma$  induced by the set $\{b\in \Gamma |\partial_\Gamma(a,b)=i\}$,
and the size of $\Gamma_1(a)$ is said to be the \textit{valency} of $a$ in $\Gamma$. 
For a connected graph $\Gamma$ of diameter $d$ and a subset of indices $I\subseteq\{0,...,d\}$ we denote by $\Gamma_{I}$ the graph on the vertex set of $\Gamma$, whose edges are the pairs of vertices $a$ and $b$ 
such that $\partial_\Gamma(a,b)\in I$. 
A graph is said to be \textit{regular}, if all its vertices  have the same valency. 
A connected  graph  $\Gamma$ of diameter $d$ is called  \textit{distance-regular}, if there are constants
 $c_i,a_i$ and $b_i$ such that for all $i\in \{0,1,\ldots,d\}$ and for each pair  of vertices $x$ and $y$ 
 such that $\partial_{\Gamma}(x,y)=i$, the following equalities hold: $c_i=|\Gamma_{i-1} (x)\cap \Gamma_1(y)|, a_i=|\Gamma_{i} (x)\cap \Gamma_1(y)|$, and $b_i=|\Gamma_{i+1} (x)\cap \Gamma_1(y)|$ (it is assumed that $b_d=c_0=0$), and, in particular, $|\Gamma_1(x)|=b_0=c_i+a_i+b_i$ (implying $\Gamma$ is regular of valency $b_0$).  The consequence $\{b_0, b_1,\ldots, b_{d-1} ; c_1,\ldots, c_d \}$ is called the \textit{intersection array} of $\Gamma$. 
  If the binary relation ``to be at distance 0 or d'' on the set of vertices of a connected graph $\Gamma$
  of diameter $d$  is an equivalence relation, then the graph $\Gamma$ is called \textit{antipodal} and 
  classes of this relation are called \textit{antipodal classes}  of $\Gamma$.
 An important subclass of the so-called imprimitive distance-regular graphs is formed by antipodal distance-regular graphs of diameter 3. The latter ones are antipodal covers of complete graphs, which emerge in various 
 geometric and combinatorial objects \cite{BCN,GH}.

A regular graph $\Gamma$ of valency $k$ on $n$ vertices is called a \textit{Deza graph} with parameters $(n,k,b,a)$
if the number of common neighbours of two distinct vertices takes on only two values $a$ or $b$ (it is assumed that $a\leq b$). 
Deza graphs were  originally invented  as a generalization of \textit{strongly regular graphs} (distance-regular graphs of diameter 2). A Deza graph is called a \textit{strictly Deza graph} if it has diameter $2$ and $a\ne b$
(the last condition implies it cannot be  strongly regular). Apart from other important classes of graphs, they  also include   the so-called \textit{divisible design graphs}, that is, the class of graphs whose every representative is regular  and admits a partition of the vertex set   into  classes of the same size  such that 
the number of common neighbors for any   two distinct vertices depends only on whether these  vertices belong to the same partition class or not  \cite{HKM}.  
For a divisible design graph, such a  partition of the vertex set is called  \textit{canonical}.

 Our group-theoretic terminology and  notation are mostly standard  and follow \cite{Asch,B11}.

The next result shows that there are infinite families of Deza graphs which are related to antipodal 
distance-regular graphs of diameter 3. 
 
\begin{lem}\label{l1}
	Let $\Gamma$ be an antipodal distance-regular graph of diameter $3$ with intersection array
	$\{k,(r-1)\mu,1;1,\mu,k\}$, where $r>2$ and $k=r\mu+1$.
	Then $\Gamma$ is a Deza graph  with parameters $$(r(k+1),  k, \mu, 0),$$  and
	$\Gamma_2$ is a Deza graph of diameter $2$ with parameters $$(r(k+1), (r-1)k, b, a),$$ where $\{a, b\}=\{(r-1)^2\mu, k(r-2)\}$. In particular, $\Gamma_2$ is a divisible design graph whose canonical partition
	coincides with 	the antipodal partition of $\Gamma$.	
\end{lem}
\begin{proof}
   
First, let us compute the values of the intersection numbers  $p^1_{22}, p^2_{22}$ and $p^3_{22}$ of $\Gamma$
(recall that  $p_{ij}^t=|\{x\in \Gamma| \partial_\Gamma(a,x)=i, \partial_\Gamma(x,b)=j\}|$  does not depend on the choice of the pair of vertices $(a,b)$ with $\partial_\Gamma(a,b)=t$).
By \cite[Lemma 4.1.7]{BCN}, we have  
\[  p^0_{22}=(r-1)k, 
	p^1_{22}=\frac{b_1^2}{c_2}=(r-1)^2\mu,  
   p^2_{22}=b_1+\frac{a_2(a_2-a_1)}{c_2}=(r-1)^2\mu, 
\]
\[    p^3_{22}=\frac{c_3(a_2+a_3-a_1)}{c_2}=k(r-2).
\]
It follows that $\Gamma_2$ is a regular graph of valency $p^0_{22}$. 
Moreover,  observe that $\Gamma_2$ has diameter 2. Indeed,  for any  distinct non-adjacent vertices $a$ and $b$
of $\Gamma_2$ we have $\partial_\Gamma(a,b)=1,3$. Besides,  since $p^1_{22}$ and $p^3_{22}$ are both non-zero, there is a vertex $x$  such that $\partial_\Gamma(a,x)=\partial_\Gamma(b,x)=2$,  and thus $\partial_{\Gamma_2}(a,b)\le 2$.
It is also  clear that $\Gamma_2$ cannot be a complete graph, which proves the required claim.
Hence any two distinct vertices of  $\Gamma_2$  have precisely $p^1_{22}=p^2_{22}$  or $p^3_{22}$ 
common neighbors, which implies that  $\Gamma_2$ is a Deza graph (which is also edge-regular with $\lambda=p^2_{22}$). 

The remaining statements follow  immediately from the definition of $\Gamma$.
\end{proof}
\begin{rem} Note that the result of   this lemma 
was initially proved in  \cite[Proposition 4.15]{HKM} in a matrix form, 
however, no explicit formulas for the  quadruple of parameters of $\Gamma_2$ 
were provided there.\end{rem}

\begin{lem}\label{l2}
	Let $\Gamma$ be an antipodal distance-regular graph of diameter $3$ with intersection array
	$\{k,(r-1)\mu,1;1,\mu,k\}$, where $r\notin\{2, \mu+2\}$ and $k=r\mu+1$,
	and put $\Phi=\Gamma_2$. Then  
	 the graph $\Omega^c$ on the vertex set of $\Phi$, whose edges are
	the pairs of distinct vertices $x$ and $y$ such that $|\Phi_1(x)\cap\Phi_1(y)|=c$, is isomorphic to the graph $K_{(k+1)\times r}$ (that is, a complete multipartite graph with $k+1$
	parts of the same size $r$) 	 if $c=(r-1)^2\mu$, while it is  a union of $k+1$ isolated $r$-cliques
	if $c=k(r-2)$.
\end{lem} 
\begin{proof}
First note that condition $r\neq \mu+2$ certifies that $\Phi$ is a strictly Deza graph. 
Then by Lemma~\ref{l1} it easy to see that $\Omega^{(r-1)^2\mu}$ is the complement graph for $\Omega^{k(r-2)}$.
Take $c=k(r-2)$. It remains to note that $x\in (\Omega^c)_1(y)$ if and only if $x\in \Gamma_3(y)$. That is, $\Omega^c=\Gamma_3$ is a union of $k+1$ isolated $r$-cliques and hence $\Omega^{(r-1)^2\mu}=\overline{\Omega^c}\cong K_{(k+1)\times r}$.
\end{proof}
\begin{rem} 
It follows that Lemma~\ref{l1} actually provides a construction of an infinite class of Deza  graphs with    imprimitive strongly regular graphs $\Omega^c$, which seems to be unnoticed before. As we will see below (see also \cite{Tsi15,Tsi17}), the structure of the automorphism group of such a Deza graph can be rather sophisticated. 
\end{rem}
\medskip

Let $G\in \{L_2(q),Sz(q),U_3(q)\}$,  where $q=2^n\ge 4$. Further we denote by   
 $\chi(G)$   the  associated prime number of $G$ in the sense by M. Suzuki, that is
 \[\chi(G)= \begin{cases}
  5,  & \mbox{if } G=Sz(q)  \\
  3, & \mbox{if } G\in \{PSL_2(q),PSU_3(q)\}
\end{cases}.\] 
A pair of involutions $\{x,y\}$  of  $G$ is called \textit{distinguished}, if $|xy|=\chi(G)$. 
Note that due to a result of Suzuki (e.g. see \cite{Suz0, Suz2}),   $G$ acts transitively on the set of ordered distinguished  pairs of involutions, and hence there is a unique conjugacy class  $\mathcal{S}$  of dihedral subgroups of order  ${2\cdot\chi(G)}$ in  $G$. In other words, a pair of involutions $\{x,y\}$  of  $G$ is distinguished if and only if $\langle x, y\rangle \in \mathcal{S}$.  
This observation, in particular, implies there is an isomorphism between $\{\chi(G)\}$-local fusion graph
of $G=PSL_2(q)$ with $q=2^n\ge 4$ and its $S_3$-involution graph, which was shown to have diameter 3 in \cite{DevGiu} and moreover, by \cite{Tsi17},  is isomorphic to a Mathon graph (for a construction of the latter, see  \cite[Proposition 12.5.3]{BCN}).
Next we reformulate   some results of  \cite{Tsi15} and \cite{Tsi17} in terms of local fusion graphs,
that will be basic to   further  arguments.
 
\begin{thm}[see {\cite{Tsi15,Tsi17}}]\label{t2}
For each  group $G\in \{PSL_2(q), PSU_3(q), Sz(q)\}$ with $q=2^n\ge 4$, its $\{\chi(G)\}$-local fusion graph
is an arc-transitive antipodal distance-regular graph of diameter $3$  with intersection array 
\begin{itemize}
\item[$(1)$] $\{q,q-2,1;1,1,q\}$  if $G=PSL_2(q)$,
\item[$(2)$] $\{q^2,q^2-q-2,1;1,q+1,q^2\}$ if $G=Sz(q)$, or
\item[$(3)$] $\{q^3,q^3-q^2-q-2,1;1,q^2+q+1,q^3\}$ if $G=PSU_3(q)$.
\end{itemize}\end{thm}

 \section{Main result}  
Now we present the main result of this paper.

\begin{thm}\label{t1} For each group $G\in \{PSL_2(q), PSU_3(q), Sz(q)\}$ with $q=2^n\ge 4$ and  for $\pi$
being the subset of all odd integers in $\omega(G)-\{2,\chi(G)\}$,
a (unique) $\pi$-local fusion graph of $G$  is a vertex-transitive edge-regular Deza graph of diameter $2$
with parameters $(v, k, b, a)$ as follows:
\begin{itemize}
\item[$(1)$] $(q^2-1, q(q-2), q(q-3), (q-2)^2)$ if $G=PSL_2(q)$,
\item[$(2)$] $((q^2+1)(q-1), q^2(q-2), (q-2)^2(q+1), q^2(q-3))$ if $G=Sz(q)$,
\item[$(3)$] $((q^3+1)(q-1), q^3(q-2), (q-2)^2(q^2+q+1), q^3(q-3))$ if $G=PSU_3(q)$.
\end{itemize}
Moreover, it is a divisible design graph, in which every class of the canonical partition is 
formed by the set of involutions of a Sylow $2$-subgroup of $G$.
\end{thm}
 
\begin{proof}
Put $\tilde\pi=\omega(G)-\{2,\chi(G)\}$. 
Let $X$ be the class of involutions of $G$ and let $\Gamma$ be the $\{\chi(G)\}$-local fusion graph of $G$. Then by Theorem~\ref{t2}
$\Gamma$ is an antipodal distance-regular graph of diameter 3 with intersection array 
$$\{q^l, (q-2)(q^l-1)/(q-1),1;1,(q^l-1)/(q-1), q^l\},$$
 where $q^l$ is exactly the size of a Sylow 2-subgroup of $G$ (so that $l\in \{1,2,3\}$).
Note that each antipodal class of $\Gamma$ is formed by the set of  (central) involutions of
 a Sylow $2$-subgroup of $G$ and there exactly $q^l+1$ such classes. 

Now define the graph $\Phi$ as the graph $\Gamma$ with  antipodal classes turned into cliques.
By construction, $\Phi$ coincides  with $\Gamma_{1,3}$. 
Let us prove that  $\Phi$ is a complement graph to  
a (unique) $\tilde\pi$-product involution graph $gF_{\tilde\pi}(G,X)$ of $G$.  
To find this, it suffices  to observe  that by  \cite{Suz0} (see also \cite{Suz2}) the order of product of any two involutions of $G$ equals 2 if and only if these involutions commute (and 
thus they belong to the same Sylow 2-subgroup of $G$).

Hence  $gF_{\tilde\pi}(G,X)=\Gamma_2$.

Therefore, by Lemma~\ref{l1}  
 $gF_{\tilde\pi}(G,X)$  is a  Deza graph (more precisely, a divisible design graph whose canonical partition coincides with antipodal partition of $\Gamma$)
of diameter 2, in which the number of common neighbors of two distinct vertices equals  $(q-2)^2(q^l-1)/(q-1)$ or $q^l(q-3)$. 
Clearly, these values are the same if and only if $l=1$ and $G=PSL_2(4)$ (yielding $gF_{\tilde\pi}(G,X)$ is isomorphic to  the Petersen graph).
Hence in all other cases $gF_{\tilde\pi}(G,X)$ is a strictly Deza graph.

Furthermore, note that the order of the product of every two non-commuting involutions of $G$ cannot be even.
On the contrary, suppose there are involutions $x$ and $y$  such that $|xy|=2m$ and 
$z=(xy)^m\ne 1$. Then $z^2=1$ and hence $z^x=z^{-1}=z=z^y$. 
But  ${\rm Syl}_2(G)$ is a TI-subset of $G$ and each Sylow 2-subgroup $S$ of $G$ acts (by conjugation) regularly 
 on ${\rm Syl}_2(G)\setminus\{S\}$ (see \cite{Suz0}), a contradiction.

Thus, we conlude that  $gF_{\tilde\pi}(G,X)$ coincides with the ${\pi}$-local fusion graph of $G$, where
${\pi}$ consists of  odd elements of $\tilde\pi$.
\end{proof}
  
\begin{rem} Except for the case $G=PSL_2(4)$, it appears that ${\pi}$-local fusion graphs of $G$ defined in Theorem~\ref{t1} are first established to be (strictly) Deza graphs in the course of this work. 
\end{rem}

 \section*{Acknowledgements}
This research was supported by the Russian Science Foundation under grant no. 20-71-00122  and performed in N.N. Krasovskii Institute of Mathematics and Mechanics of the Ural Branch of the Russian Academy of Sciences.

\end{document}